\pdfoutput=1
\documentclass[a4paper, 10pt]{amsart}

\usepackage[protrusion=true,expansion=true]{microtype}

\usepackage{float}
\usepackage{physics}
\usepackage{fancyhdr}
\usepackage[utf8]{inputenc}
\usepackage{graphicx}
\usepackage{wrapfig}
\usepackage{mathtools}
\usepackage{adjustbox}
\usepackage{leftidx}
\usepackage{enumitem}
\usepackage{stackrel}
\usepackage{cite}

\usepackage{mathpazo}
\usepackage{mathrsfs}
\usepackage[T1]{fontenc}
\usepackage{amsmath}
\usepackage{amssymb}
\usepackage{hyperref}
\usepackage{cleveref}
\usepackage{comment}
\usepackage{color}
\usepackage{tikz}
\usepackage{tikz-cd}
\usepackage{xparse}
\usetikzlibrary{matrix,arrows,decorations.pathmorphing}

\newtheorem{example}{Example}[section]
\newtheorem{theorem}{Theorem}[section]
\newtheorem*{theorem*}{Theorem}

\newtheorem{corollary}{Corollary}[section]

\newtheorem{lemma}{Lemma}[section]

\crefname{lemma}{Lemma}{lemma}
\crefname{remark}{Remark}{remark}
\crefname{corollary}{Corollary}{corollary}
\crefname{theorem}{Theorem}{theorem}
\crefname{proposition}{Proposition}{proposition}
\crefname{example}{Example}{example}
\crefname{definition}{Definition}{definition}
\crefname{notation}{Notation}{notation}
\crefname{appendix}{Appendix}{appendix}
\crefname{section}{Section}{section}

\newcommand{\AAA}{\mathfrak A}
\newcommand{\BBB}{\mathcal B}

\newcommand{\HHH}{\mathcal H} 

\newcommand{\PPP}{\mathscr P}

\newcommand{\C}{\mathbb C} 
\newcommand{\R}{\mathbb R}  
\newcommand{\N}{\mathbb N} 

\newcommand{\titleinfo}{Wigner-Type Theorem on transition probability preserving maps in semifinite factors}
\newcommand{\titleinfoshort}{Wigner-Type Theorem}
\newcommand{\authorinfo}{Wenhua Qian$^1$, Liguang Wang$^2$, Wenming Wu$^1,^\dag$, Wei Yuan$^{3,4}$}

\begin{document}

\title{\LARGE\textbf{\titleinfo}}
\author{\large\textsc{\authorinfo}}

\address{1 School of Mathematical Sciences, Chongqing Normal University, Chongqing 401331, China}
\email{whqian86@163.com, wuwm@amss.ac.cn}

\address{2 School of Mathematical Sciences, Qufu Normal University, Shandong 273165, China}
\email{wangliguang0510@163.com}

\address{3 Institute of Mathematics, Academy of Mathematics and Systems Science \\
Chinese Academy of Sciences, Beijing 100190, China}
\address{4 School of Mathematical Sciences, University of Chinese Academy of Sciences,
Beijing 100049, China}
\email{wyuan@math.ac.cn}

\date{}

\begin{abstract}
The Wigner's theorem, which is one of the cornerstones of the mathematical formulation of
quantum mechanics, asserts that every symmetry of quantum system is unitary or anti-unitary.
This classical result was first given by Wigner in 1931. Thereafter it has been proved and
generalized in various ways by many authors. Recently, G. P. Geh\'{e}r extended Wigner's and
Moln\'{a}r's theorems and characterized the transformations on the Grassmann space of all
rank-$n$ projections which preserve the transition probability. The aim of this paper is to provide a
new approach to describe the general form of the transition probability preserving
(not necessarily bijective) maps between Grassmann spaces. As a byproduct, we are able to generalize
the results of Moln\'{a}r and G. P. Geh\'{e}r.
\end{abstract}

\thanks{The fourth author is supported by the Youth Innovation Promotion Association, CAS.}
\thanks{$^\dag$ Corresponding author.}
\maketitle

\section{Introduction and Statment of the Main Result}

The celebrated Wigner’s Symmetry Representation Theorem guarantees that every symmetry's actions
on states is induced by a unitary or an anti-unitary operator. If we identify the set of pure states
with the Grassmann space $\PPP_1$ of all rank-one orthogonal projections in $\BBB(\HHH)$, where
$\BBB(\HHH)$ denotes the algebra of all bounded linear opeartor on a Hilbert space $\HHH$, then
Wigner's theorem can be formulated mathematically as follows:
\begin{theorem*}[E. P. Wigner \cite{W}]
Let $\varphi: \PPP_1 \rightarrow \PPP_1$ be a surjective map. If $\varphi$ preserves the
transition probability, i.e., $Tr(PQ)= Tr(\varphi(P)\varphi(Q))$ for
every $P, Q \in \PPP_1$, where $Tr$ is the canonical trace on $\BBB(\HHH)$,
then there is an either unitary or anti-unitary
$U \in \BBB(\HHH)$ such that $\varphi(P) = UPU^*$.
\end{theorem*}
This result is first proved by Wigner in \cite{W}. Thereafter, it has been proved and generalized
in various ways by many authors such as \cite{BJM, BMS, MG, CW, GS, G2017, M1, M2, M4, M5,M6, S0, S, ES2},
to mention but a few. The aim of the present paper is to
extend Wigner's theorem and characterize the transition probability preserving maps
between more Grassmann spaces of projections in semifinite factors.
In order to achieve the goal, we first reformulate Wigner's theorem in the language of von Neumann algebras.

Recall that a (concrete) von Neumann algebra is a *-subalgera of $\BBB(\HHH)$ that is closed in the
weak-operator topology. If a von Neumann algebra $\AAA$ endowed with a normal faithful tracial
weight $\tau$, then $\AAA$ is called semifinite. A factor is a von Neumann algebra
whose center consists of multiples of the identity. A factor is said to be of type I if
it has a minimal projection. And a type II factor is a semifinite factor which contains no minimal projections.
We direct the reader to \cite{BB, KRII, Tak} for a general reference on the theory of von Neumann algebras.

Let $\AAA$ be a semifinite factor with a normal faithful tracial weight $\tau$ (it is
unique up to scalar multiplication) and $\PPP$ be the set of all orthogonal projections in $\AAA$.
Denote by $\PPP_c$ the Grassmann space of
all projections in $\AAA$ with the trace value $c$, i.e., $\PPP_c =
\{P \in \PPP: \tau(P) = c\}$, $c \in [0, \tau(I)]$. A map $\varphi: \PPP_c \rightarrow \PPP_c$
is transition probability preserving if $\tau(PQ)= \tau(\varphi(P)\varphi(Q))$ for
every $P, Q \in \PPP_c$.

Note that $\BBB(\HHH)$ is a semifinite type I factor. And
Wigner's theorem asserts that if $\AAA \cong \BBB(\HHH)$, then
every surjective transition probability preserving map $\varphi: \PPP_1 \rightarrow \PPP_1$
is induced by either a unitary or an anti-unitary.

Recently, G. P. Geh\'{e}r generalized Wigner's and Moln\'{a}r's theorem \cite{W,M2,M6} and prsented the
following characterisation of transition probability preserving maps on Grassmann
space $\PPP_n$ for $\AAA \cong \BBB(\HHH)$ and $2n \neq \dim\HHH$ in \cite{G2017}:
every (not necessarily bijective) transition probability preserving map $\varphi: \PPP_n \rightarrow \PPP_n$
is induced by a linear or conjugate-linear isomegry $U: \HHH \rightarrow \HHH$, i.e.,
$\varphi(P) = UPU^*$ for all $P \in \PPP_n$.\footnote{The case
when $2n = \dim \HHH$ is also handled in \cite{G2017}.}

A map $\sigma: \BBB(\HHH_1) \rightarrow \BBB(\HHH_2)$ is said to be implemenyed by a unitary (resp. by an
anti-unitary) $U: \HHH_1 \rightarrow \HHH_2$ if $\sigma(A) = UAU^*$ (resp. $\sigma(A) = UA^*U^*)$) for all
$A \in \BBB(\HHH_1)$. Since every *-isomorphism (resp. *-anti-isomorphism) $\sigma$ form
$\BBB(\HHH_1)$ to $\BBB(\HHH_2)$ is implemented by a unitary
(resp. by an anti-unitary) by \cite[Theorem 4.27]{AS}, the result in \cite{G2017} can be restated as:
\begin{theorem*}[G. P. Geh\'{e}r \cite{G2017}]
    Let $\AAA = \BBB(\HHH)$ where $\HHH$ is a Hilbert space with $\dim \HHH \geq 2$. Assume that
    $2n \neq \dim \HHH$. For every transition probability preserving map $\varphi: \PPP_n \rightarrow
\PPP_n$, there exists an orthogonal projection
$E \in \AAA = \BBB(\HHH)$ and a *-isomorphism or a *-anti-isomorphism
$\sigma: \AAA \rightarrow E\AAA E$ such that $\varphi(P) = \sigma(P)$ for all $P\in\PPP_n$.
\end{theorem*}

At this stage one is tempted to conjecture that the same result holds for every transition probability preserving map between the Grassmann space of projections in
any semifinite factor with a fixed trace value. However, as showed by the example below,
this naive generalization of the Wigner's theorem is doomed to fail.

\begin{example}
    Let $\AAA= \mathcal{R} \otimes M_2(\C)$ where $\mathcal{R}$ is the hyperfinite II$_1$ factor.
    Let $\mathcal{R}_0 \subsetneq \mathcal{R}$ be a hyperfinite subfactor of $R$.
    Since $\AAA$ is also the hyperfinite II$_1$ factor, there exist a *-isomorphism $\varphi_1$ and a
    *-anti-isomorphism $\varphi_2$ from $\AAA$ to $\mathcal{R}_0$. Then the map
    \begin{align*}
        \varphi: P \in \PPP_{c} \mapsto
        \begin{pmatrix}
            \varphi_1(P) & \\
            & \varphi_2(P) \\
        \end{pmatrix} \in \PPP_c
    \end{align*}
    preserves the transition probability.
\end{example}

Keeping in mind the example above, we formulate the main result of the paper as below.

\begin{theorem} \label{main_result}
Let $\AAA$ be a semifinite factor with a faithful normal tracial weight $\tau$ and
    $c \in (0, \tau(I)) \setminus \{\tau(I)/2\}$.
    If the map $\varphi:\PPP_c\rightarrow \PPP_c$ satisfies
\begin{align*}
    \tau(\varphi(P)\varphi(Q))=\tau(PQ), \quad \forall P,Q\in\PPP_c,
\end{align*}
then there exist two orthogonal projections $E_1$ and $E_2$ in $\AAA$ and
a $\sigma$-weak continuous *-homomorphism $\sigma_1: \AAA \rightarrow E_1 \AAA E_1$ and
a $\sigma$-weak continuous *-anti-homomorphism $\sigma_2: \AAA \rightarrow E_2 \AAA E_2$ such that
\begin{align*}
    \tau = \tau \circ (\sigma_1 + \sigma_2),
\end{align*}
    and $\varphi(P) = \sigma_1(P) + \sigma_2(P)$ for every $P \in \PPP_c$.
Furthermore,
\begin{enumerate}
    \item $\sigma_i(I) = E_i$, $i = 1,2$. And $E_1 +E_2 = I$ if $\AAA$ is a finite factor.
    \item If $\AAA$ is type I, then $\sigma_1$ or $\sigma_2$ is the zero map.
\end{enumerate}
\end{theorem}

In order to prove this result, we first characterize the transition probability preserving maps
between subsets of $\PPP$ whose linear span contains $\PPP$ (see \cref{thm_lemma} for details)
in \cref{sec2}. Next, we show that the maps considered in \cref{main_result} preserve the relation
of commutativity and inclusion between the elements of $\PPP_c$ (see \cref{P_Q_comm_general_case} and
\cref{include_two_case}) in \cref{aux}. After these preparations, we complete the proof of \cref{main_result}
in \cref{main}.

\section{The transformations from $\PPP_{S}$ to $\PPP_{S}$ preserving transition probability}\label{sec2}

In this section, we assume that $\AAA$ is a semifinite factor with
a fixed normal faithful tracial weight $\tau$. Let $\PPP$ be the set of all projections in $\AAA$. If $S$ is a subset of $[0, \infty)$, we use
$\PPP_S$ to denote the subset of $\PPP$:
\begin{align*}
    \PPP_{S} = \{P \in \PPP: \tau(P) \in S\}.
\end{align*}
For $S = \{c\}$, where $c \in [0, \infty)$, we simplify notation and write $\PPP_{c}$ for
$\PPP_{\{c\}}$.

Recall that a map $\varphi: \PPP_{S} \rightarrow \PPP_{S}$ preserves the transition probability if
\begin{align*}
    \tau(\varphi(P) \varphi(Q)) = \tau(PQ), \quad \forall P, Q \in \PPP_{S}.
\end{align*}
Note that $P\perp Q$ if and only if $PQ=QP=0$ and the following result is immediate.

\begin{lemma}\label{pre_orth}
    Let $\varphi: \PPP_{S} \rightarrow \PPP_{S}$ be a map
    preserving the transition probability. Then $\varphi$ preserves the orthogonality
    in both directions, i.e., $P \bot Q$ if and only if
    $\varphi(P) \bot \varphi(Q)$.
\end{lemma}

Let $E_1$ and $E_2$ be two orthogonal projections in $\AAA$. If there exist a *-homomorphism $\sigma_1:
\AAA \rightarrow E_1 \AAA E_1$ and a *-anti-homomorphism $\sigma_2: \AAA \rightarrow E_2 \AAA E_2$ such
that
\begin{align*}
    \tau = \tau \circ (\sigma_1 + \sigma_2),
\end{align*}
then it is not hard to check that the map,
\begin{align*}
    \varphi: P \in \PPP_{S} \rightarrow \sigma_1(P) + \sigma_2(P) \in \PPP_{S},
\end{align*}
preserves the transition probability. Conversely, we show that the transition probability preserving map
$\varphi: \PPP_{S} \rightarrow \PPP_S$ always arises from a sum of a *-homomorphism and a *-anti-homomorphism of
$\AAA$, where $S$ is a subset of $[0, \infty)$ such that
every projection $P \in \PPP$ is a sum of a family of mutaually orthogonal
projections in $\PPP_S$. The case when $S = [0, \infty)$ and $\varphi$ is bijective
is handled by L. Moln\'{a}r in \cite{M1}. By modifying the argument used in \cite{M1},
we can prove \cref{thm_lemma} below.

Throughout the rest of this section, we assume that $\PPP_S$ is a subset of $\PPP$ such that
every projection $P \in \PPP$ is a sum of a family of mutaually orthogonal projections in $\PPP_S$.

\begin{lemma}\label{ext_Phi_orth}
    If $\varphi$ is a transformation from $\PPP_{S}$ to $\PPP_{S}$ preserving the
    transition probablity, then $\varphi$ can be extended to a transformation $\Phi: \PPP \rightarrow \PPP$
    such that
    \begin{align*}
        \Phi(\sum_{\alpha} E_\alpha) = \sum_{\alpha} \Phi(E_\alpha)
   \end{align*}
    whenever $\{E_\alpha\}$ is an orthogonal family of projections in $\PPP$.
    Furthermore, if $\varphi$ is surjective, then $\Phi$ is also surjective.
\end{lemma}

\begin{proof}
    Let $P \in \PPP$. There exists a family of pairwise orthogonal projections $\{P_{\alpha}\}_{\alpha}$ in
    $\PPP_{S}$ such that $P = \sum_{\alpha}P_{\alpha}$. We define
    \begin{align*}
        \Phi(P) = \sum_\alpha \varphi(P_\alpha).
    \end{align*}
    By \cref{pre_orth}, $\Phi(P) \in \PPP$. If $\{Q_{\beta}\} \subset \PPP_{S}$ also satisfies
    $P = \sum_{\beta}Q_\beta$, then for every $E \in \PPP_{S}$, we have
    \begin{align*}
        \tau([\sum_{\alpha} \varphi(P_\alpha)]\varphi(E)) = \tau([\sum_{\alpha}P_\alpha]E)
        =\tau([\sum_{\beta}Q_\beta]E) = \tau([\sum_{\beta} \varphi(Q_\beta)]\varphi(E)).
    \end{align*}
    Then it is clear that
    \begin{align*}
        \varphi(P_\alpha) \leq \sum_{\beta}\varphi(Q_\beta),
        \quad \varphi(Q_\beta) \leq \sum_{\alpha}\varphi(P_{\alpha}),
    \end{align*}
    and $\sum_{\alpha} \varphi(P_\alpha)
    = \sum_{\beta} \varphi(Q_\beta)$. Thus $\Phi$ is well-defined.

    Now assume that $\{E_\alpha\}$ is an orthogonal family of projections in $\PPP$.
    For each $E_\alpha$, there exists a family of pairwise orthogonal projections $\{P_{\alpha}^{\beta}\}_{\beta}
    \subset \PPP_{S}$ such that $E_\alpha = \sum_\beta P_{\alpha}^{\beta}$.
    By the definition of $\Phi$ and note that $P_{\alpha_1}^{\beta_1} \bot P_{\alpha_2}^{\beta_2}$ unless
    $\alpha_1 = \alpha_1$ and $\beta_1 = \beta_2$, we have
    \begin{align*}
        \Phi(\sum_{\alpha} E_\alpha) = \Phi(\sum_{\alpha, \beta}P_{\alpha}^{\beta}) =
        \sum_{\alpha, \beta} \varphi(P_{\alpha}^{\beta}) = \sum_{\alpha} \Phi(E_\alpha).
    \end{align*}
    Finally, since any projection in $\AAA$ is the sums of projections in $\PPP_{S}$,
    the surjectivity of $\varphi$ implies the surjectivity of $\Phi$.
\end{proof}

\begin{lemma}\label{ext_C_hom}
    Let $\Phi$ be the map defined in \cref{ext_Phi_orth}.
    Then $\Phi$ can be extended to a C$^*$-homomorphism (\cite{ES}) from $\AAA$
    into $\AAA$, i.e., a linear self-adjoint map such that $\Phi(H^2) = \Phi(H)^2$ for every
    self-adjoint operator $H \in \AAA$.
\end{lemma}

\begin{proof}
    We claim that $\Phi$ can be extended to a bounded linear operator from $\AAA$ into $\AAA$.
    If $\AAA$ is not type I$_2$ factor, then \cite[Theoerem A]{BW} implies that the orthoadditive
    transformation $\Phi$ has a unique extension to a bounded linear operator
    from $\AAA$ into $\AAA$. Assume now that $\AAA \cong M_2(\C)$.
    Then $\PPP_{S}$ contains the all rank-one projections.
    Since every operator in $\AAA$ is a linear combination of rank-one projections,
    it is not hard to check that the following gives a bounded extension of $\Phi$:
    \begin{align*}
        \sum_{i=1}^n \lambda_i P_i \in M_2(\C) \mapsto \sum_{i=1}^n \lambda_i\varphi(P_i) \in M_2(\C),
    \end{align*}
    where $P_i$ is a rank-one projection, $i = 1, \ldots, n$.
    Thus claim is proved. To simplify the notation, we use $\Phi$ to denote its extension as well.

    For every family of pairwise orthogonal
    projections $\{P_1, \ldots, P_n\}$ and $\lambda_1, \ldots \lambda_n \in \R$, we have
    \begin{align*}
        \Phi(\sum_{i=1}^{n} \lambda_i P_i)^2 = (\sum_{i=1}^n \lambda_i \Phi(P_i))^2
        = \sum_{i=1}^n \lambda_i^2 \Phi(P_i) = \Phi((\sum_{i=1}^n \lambda_i P_i)^2).
    \end{align*}
    Since $\Phi$ is bounded, we have $\Phi(H)$ is self-adjoint and
    $\Phi(H^2) = \Phi(H)^2$ for every self-adjoint operator $H \in \AAA$.
    Therefore, $\Phi$ is a C$^*$-homomorphism (actually it is not hard to show that $\Phi$ is
    a Jordan *-homomorphism \cite[Proof of Theorem 2]{M0}).
\end{proof}

\begin{theorem}\label{thm_lemma}
    Let $S$ be a subset of $[0, \infty)$ such that every projection $P \in \PPP$ is a sum of a
    family of mutaually orthogonal projections in $\PPP_S$ and
    $\varphi:\PPP_{S} \rightarrow \PPP_{S}$ be a map preserving the transition probability.
    Let $\Phi$ be the C$^*$-homomorphism defined in \cref{ext_C_hom} and
    $\BBB$ be the C$^*$-algebra generated by $\Phi(\PPP)$. Then
    $\BBB$ is a von Neumann subalgebra of $\AAA$. Furthermore,
    there exist two orthogonal central projections $E_1$ and $E_2$ in $\BBB$ with $E_1 +E_2 = I_{\BBB}$
    (here $I_{\BBB}$ is the identity of $\BBB$ and we use $I$ to denote the identity of $\AAA$)
    and a $\sigma$-weak continuous *-homormorphsim
    $\sigma_1: \AAA \rightarrow \BBB E_1$ and a $\sigma$-weak continuous *-anti-homomorphism
    $\sigma_2: \AAA \rightarrow \BBB E_2$ such that:
    \begin{enumerate}
        \item $\varphi(P) = \sigma_1(P) + \sigma_2(P)$ for every $P \in \PPP_{S}$.
        \item $\tau = \tau \circ (\sigma_1 + \sigma_2)$. Therefore $E_1 + E_2 = I$ if
            $\AAA$ is a finite factor.
        \item If $\AAA$ is a type I factor, then $\sigma_1$ or $\sigma_2$ is the zero map.
    \end{enumerate}
\end{theorem}

\begin{proof}
    \cite[Lemma 3.2]{ES} implies that there exist central projections $E_i$ of $\BBB$ with $E_1 + E_2 = I_\BBB$
    such that the map $\sigma_1: A \in \AAA
    \rightarrow \Phi(A)E_1$ (resp. $\sigma_2: A \rightarrow \Phi(A)E_2$) is a *-homomorphism (resp.
    *-anti-homomorphism) and $\Phi(A) = \sigma_1(A) + \sigma_2(A)$ for every $A \in \AAA$.

    We now show that $\sigma_1$ and $\sigma_2$ are $\sigma$-weak continous.
    For every state $\rho \in \AAA_*$ where $\AAA_*$ is the predual of $\AAA$, $\rho \circ \sigma_i$ is
    completely additive by \cite[Theorem III.2.1.4]{BB} and \cref{ext_Phi_orth}. Thus
    $\rho \circ \sigma_i \in \AAA_*$. Since every normal linear functional is a linear
    combination of normal states, we have $\sigma_i$ is $\sigma$-weakly continuous.
    Note that $E_i \in \BBB$, $\BBB E_i = \sigma_i(\AAA)$ is a von Neumann subalgebra of $\AAA$, $i=1,2$
    (see \cite[Proposition 7.1.15]{KRII}) and $\BBB$ is also a von Neumann algebra.

    Since $\varphi$ preserves the transition probabilty, we have
    \begin{align*}
        \tau(P) = \tau(\varphi(P)) = \tau(\sigma_1(P) + \sigma_2(P)), \quad \forall P \in \PPP_{S}.
    \end{align*}
    By the defintion of $\Phi$, $\tau(Q) = \tau(\Phi(Q)) = \tau(\sigma_1(Q)+\sigma_2(Q))$ for
    every $Q \in \PPP$. Note that the finite linear combinations of finite projections in $\AAA$ is
    a $\sigma$-weakly dense *-subalgebra of $\AAA$. Then \cite[Vol. II, VIII. Proposition 3.15]{Tak}
    implies that $\tau = \tau \circ (\sigma_1 + \sigma_2)$.
    Since $\sigma_i(I) = E_i$, $i = 1,2$, $\tau(E_1 + E_2) = \tau(I)$ implies that $E_1 + E_2 = I$
    if $\AAA$ is a finite von Neumann algebra.

    Assume that $\AAA$ is a type I factor. Note that $\sigma_1(P)+ \sigma_2(P) = \Phi(P) $ is
    rank one if $P$ is a rank one projection. Since $\sigma_1$ and
    $\sigma_2$ are $\sigma$-weak continuous, we have $\sigma_1$ or $\sigma_2$ is the zero map.
\end{proof}

\section{Auxiliary results for maps $\varphi:\PPP_c \rightarrow \PPP_c$, $c \in (0, \tau(I)/2)$} \label{aux}

From now on, we assume that $\AAA$ is a semifinite factor which is not type I$_2$ and
$\varphi: \PPP_c \rightarrow \PPP_c$ is a map
preserving the transition probability, where $c \in (0, \tau(I)/2)$ and $\PPP_c \neq \emptyset$ or $\{0\}$.

The following two results must be well-known to experts (\cref{P_Q_commut_iff_con} is essentially
\cite[Lemma 5]{M2}) and we sketch the proof for the reader's convenience.

\begin{lemma}\label{P_orth_incldue}
    Let $P_1, P_2, Q \in \PPP_c$. If $P_1 \bot P_2$, then $Q \leq P_1 + P_2$ if and only
    if $\varphi(Q) \leq \varphi(P_1)+ \varphi(P_2)$.
\end{lemma}

\begin{proof}
    By \cref{pre_orth}, $\varphi(P_1)+ \varphi(P_2)$ is a projection. Note that
    \begin{align*}
        \tau(\varphi(Q)[\varphi(P_1)+ \varphi(P_2)]\varphi(Q)) = \tau(Q[P_1+P_2]Q).
    \end{align*}
    Thus $Q \leq P_1 + P_2$ if and only if $\varphi(Q) \leq \varphi(P_1)+ \varphi(P_2)$.
\end{proof}

\begin{lemma}\label{P_Q_commut_iff_con}
    Let $P, Q \in\PPP$ and $E$ be the range projection of $PQP$. Assume that $\tau(P) < \infty$.
    Then $\tau(PQ) \leq \tau(E)$ and $\tau(PQP) = \tau(E)$ if and only if $PQ =QP$.
\end{lemma}

\begin{proof}
   Note that $0 \leq PQP \leq E$. Thus $\tau(PQP) \leq \tau(E)$.
  It is clear that $PQ=QP$ if and only if $PQP=E$. Thus $PQ=QP$ if and only if
  $\tau(PQP) = \tau(E)$.
\end{proof}

Let $Q, P_1, P_2 \in \PPP_c$ such that $Q \leq P_1 \vee P_2$ and $P_1 P_2 = P_2P_1$.
If $\AAA$ is a type I factor and $\PPP_c$ is the set of the minimal projections in $\AAA$, then
 $P_1 = P_2$ or $P_1 \bot P_2$. Therefore \cref{pre_orth} implies that
 $\varphi(P_1)\varphi(P_2) = \varphi(P_2) \varphi(P_1)$, and
$\varphi(Q) \leq \varphi(P_1) \vee \varphi(P_2)$ by \cref{P_orth_incldue}.
For the rest of this section, we will prove in \cref{P_Q_comm_general_case} and \cref{include_two_case}
that the same result holds for every semifinite factor $\AAA$ which is not type I$_2$ and
transition probability preserving map $\varphi:\PPP_c \rightarrow \PPP_c$, $c \in (0, \tau(I)/2)$.

By the discussion above, we will assume that $\PPP_c$ is not the set of minimal projections of $\AAA$
for the rest of the section.

\begin{lemma}\label{comm_induct_step}
    Let $c_0 = \min(\tau(I) - 2c, c/2)$ and $P, Q$ be two commutative
    projections in $\PPP_c$. If $\tau(PQ) \geq c-c_0$ or $\tau(PQ) \leq c_0$,
    then $\varphi(P)\varphi(Q)= \varphi(Q)\varphi(P)$. In particular, if $\frac{5c}{2} \leq \tau(I)$, then
    $PQ = QP$ implies that $\varphi(P)\varphi(Q) =\varphi(Q)\varphi(P)$ for every $P$, $Q \in \PPP_c$.
\end{lemma}

\begin{proof}
   Let $P' \in \PPP_c$ such that $P \bot P'$ and $Q \leq P + P'$.
   By \cref{P_orth_incldue}, it is clear that $\varphi(P)\varphi(Q) = \varphi(Q)\varphi(P)$
   if and only if $\varphi(P')\varphi(Q) = \varphi(Q)\varphi(P')$. Thus we may assume that
   $\tau(PQ) \geq c-c_0$.

    Since $2c + c_0 \leq \tau(I)$, we can
    choose $Q' \in \PPP_c$ such that $Q' \bot (P \vee Q)$, $Q'P' = P'Q'$ and $\tau(P'Q') = \tau(PQ)$.
    Note that $\tau(P'(Q+Q')) = c$ and $P' \leq Q+Q'$. By \cref{pre_orth},
    \begin{align*}
        0=\varphi(Q')\varphi(Q) =  \varphi(Q')[\varphi(P)+ \varphi(P')]\varphi(Q) =
        \varphi(Q')\varphi(P')\varphi(Q).
    \end{align*}
    Thus $[\varphi(P')\varphi(Q')\varphi(P')][\varphi(P')\varphi(Q)\varphi(P')] = 0$.

    Let $E$ and $F$ be the range projections of $\varphi(P')\varphi(Q)\varphi(P')$ and
    $\varphi(P')\varphi(Q')\varphi(P')$ respectively.
    $[\varphi(P')\varphi(Q')\varphi(P')][\varphi(P')\varphi(Q)\varphi(P')] = 0$ implies that
    $EF = 0$. If $\varphi(P')\varphi(Q) \neq \varphi(Q)\varphi(P')$,
    \cref{P_Q_commut_iff_con} implies that $\tau(E) > \tau(\varphi(P')\varphi(Q)) = \tau(P'Q)$.
    Note that $\tau(F) \geq \tau(P'Q')$ and $E \vee F = E + F \leq \varphi(P')$. Thus we have
    \begin{align*}
        c = \tau(P'(Q+Q')) < \tau(E+F) \leq \tau(P') = c.
    \end{align*}
    Hence $\varphi(P)\varphi(Q)$ must equals $\varphi(Q)\varphi(P)$.
\end{proof}

\begin{lemma}\label{comm_induction}
    Let $c_1 \in (0, c)$ such that $\PPP_{c_1} \neq \emptyset$. Suppose that
    for every $P$, $Q \in \PPP_c$ such that $PQ = QP$ and
    $\tau(PQ) \leq c_1$, we have $\varphi(P) \varphi(Q) = \varphi(Q)\varphi(P)$.
    Then $PQ = QP$ and $\tau(PQ) \leq 2c_1$ also implies that $\varphi(P) \varphi(Q) = \varphi(Q)\varphi(P)$.
\end{lemma}

\begin{proof}
    We choose two orthogonal projections $P_1$, $P_2$ in $\PPP_c$ such that
    $P_iP = PP_i$, $P_iQ = QP_i$, $i = 1, 2$ and $PQ = PP_2 + QP_1$,
    \begin{align*}
        P \leq P_1 + P_2,\quad Q \leq P_1 + P_2, \quad \tau(P_1Q) \leq c_1, \quad \tau(P_2 P) \leq c_1.
    \end{align*}
    By \cref{P_orth_incldue} and the assumption, we have
    \begin{align*}
        \varphi(P_i)\varphi(P) = \varphi(P)\varphi(P_i), \quad \varphi(P_i)\varphi(Q) = \varphi(Q)\varphi(P_i),
        \quad i = 1,2.
    \end{align*}
    Thus $\varphi(P_1)\varphi(Q)$ and $\varphi(P_2)\varphi(P)$ are two projections whose traces are
    $\tau(P_1Q)$ and $\tau(P_2P)$ respectively.
    Let $E_i$ be the range projection of $\varphi(P)\varphi(Q)\varphi(P)\varphi(P_i)$, $i=1,2$.
    It is clear that $E_1E_2 = 0$ and
    \begin{align*}
        \tau(E_1) \leq \tau(\varphi(Q)\varphi(P_1)) = \tau(QP_1), \quad \tau(E_2) \leq \tau(\varphi(P)\varphi(P_2))
        = \tau(PP_2).
    \end{align*}
    Note that the range projection of $\varphi(P)\varphi(Q)\varphi(P)$ is $E_1 + E_2$ and
    \begin{align*}
        \tau(\varphi(P)\varphi(Q)) \leq \tau(E_1 + E_2) \leq \tau(QP_1)+ \tau(PP_2) = \tau(PQ).
    \end{align*}
    Therefore $\tau(E_1 + E_2) = \tau(\varphi(P)\varphi(Q))$ and \cref{P_Q_commut_iff_con} implies that
    $\varphi(P)\varphi(Q) = \varphi(Q)\varphi(P)$.
\end{proof}

\begin{theorem}\label{P_Q_comm_general_case}
    Let $P, Q \in \PPP_c$. If $PQ =QP$ then $\varphi(P)\varphi(Q)= \varphi(Q)\varphi(P)$.
\end{theorem}

\begin{proof}
    If $5c/2 \leq \tau(I)$, the theorem is proved by \cref{comm_induct_step}.
    Assume $2c < \tau(I) < 5c/2$, let $c_0 = \tau(I) - 2c < c/2$. Note that
    $\PPP_{c_0} \neq \emptyset$. Then \cref{comm_induction} implies the result.
\end{proof}

\begin{lemma}\label{include_comm_ind_step}
    Let $c_0 = \min(\tau(I)-2c, c)$ and $P_1, P_2, Q \in \PPP_c$.
    Suppose that $P_1 P_2 = P_2 P_1$ and $\tau(P_1P_2) \leq c_0$.
    Then $Q \leq P_1 \vee P_2$ implies $\varphi(Q) \leq \varphi(P_1) \vee \varphi(P_2)$.
\end{lemma}

\begin{proof}
   Since $2c + \tau(P_1P_2) \leq \tau(I)$, we can choose $P_3, P_4 \in \PPP_c$ such that
    $\{P_1, P_2, P_3, P_4\}$ are mutually commutative and
    \begin{align*}
        P_3 \bot P_1, \quad P_4 \bot P_1, \quad P_2 \leq P_1 + P_3, \quad P_2 \leq P_1 +P_4,
        \quad P_3 P_2 = P_4P_2 = P_3P_4.
    \end{align*}
   By \cref{P_orth_incldue}, we have
    \begin{align*}
        \varphi(P_2) \leq \varphi(P_1)+ \varphi(P_3), \quad \varphi(P_2) \leq \varphi(P_1)+ \varphi(P_4).
    \end{align*}
    Since $\varphi(P_1) \vee \varphi(P_2) \leq
    [\varphi(P_1)+ \varphi(P_3)] \wedge [\varphi(P_1)+ \varphi(P_4)]$, \cref{P_Q_comm_general_case} implies
    that
    \begin{align*}
        \tau(\varphi(P_1) \vee \varphi(P_2)) &= 2c - \tau(P_1 P_2)= c + \tau(P_3P_4)\\
        &=\tau([\varphi(P_1)+ \varphi(P_3)] \wedge [\varphi(P_1)+ \varphi(P_4)]).
    \end{align*}
    Thus $\varphi(P_1) \vee \varphi(P_2)= [\varphi(P_1)+ \varphi(P_3)] \wedge [\varphi(P_1)+ \varphi(P_4)]$.
    By \cref{P_orth_incldue},
    \begin{align*}
        \varphi(Q) \leq [\varphi(P_1)+ \varphi(P_3)] \wedge [\varphi(P_1)+ \varphi(P_4)]
        =\varphi(P_1) \vee \varphi(P_2).
    \end{align*}
\end{proof}

\begin{lemma}\label{include_comm_induction}
    Let $c_1 \in (0, c)$ such that $\PPP_{c_1} \neq \emptyset$. Suppose that for every
    $P_1, P_2, Q \in \PPP_c$ such that $P_1 P_2 = P_2 P_1$, $\tau(P_1P_2) \leq c_1$
    and $Q \leq P_1 \vee P_2$, we have $\varphi(Q) \leq \varphi(P_1) \vee \varphi(P_2)$. Then
    $P_1 P_2 = P_2 P_1$, $\tau(P_1P_2) \leq 2c_1$ and $Q \leq P_1 \vee P_2$ also implies that
    $\varphi(Q) \leq \varphi(P_1) \vee \varphi(P_2)$.
\end{lemma}

\begin{proof}
    We assume that $P_1 \neq P_2$ and
    choose $Q_1, Q_2 \in \PPP_c$ such that $\{P_1, P_2, Q_1, Q_2\}$ is a family of mutually
    commutative projections, $Q_1 \bot Q_2$, $P_1 \vee P_2 \leq Q_1+Q_2$, $P_1Q_1 + P_2Q_2 = P_1P_2$,
    \begin{align*}
        &\tau(P_1Q_1) = a_1, \quad P_2 \leq P_1 \vee Q_1,\\
        &\tau(P_2Q_2) = a_2, \quad P_1 \leq P_2 \vee Q_2,
    \end{align*}
    where $0< a_i \leq c_1$ and $a_1 + a_2 = \tau(P_1P_2)$. Therefore, we have
    \begin{align*}
        \varphi(P_1) \vee \varphi(P_2) \leq [\varphi(P_1) \vee \varphi(Q_1)] \wedge
        [\varphi(P_2) \vee \varphi(Q_2)].
    \end{align*}
    And \cref{P_Q_comm_general_case} implies that
    \begin{align*}
        \varphi(Q_1)[I-\varphi(P_2)] \bot [\varphi(P_2) \vee \varphi(Q_2)], \quad
        \varphi(Q_2)[I-\varphi(P_1)] \bot [\varphi(P_1) \vee \varphi(Q_1)],
    \end{align*}
    and
    \begin{align*}
        \tau(\varphi(Q_1)[I-\varphi(P_2)]) = a_2, \quad \tau(\varphi(Q_2)[I-\varphi(P_1)) = a_1.
    \end{align*}
    Since $\varphi(Q_1)[I-\varphi(P_2)] \bot  \varphi(Q_2)[I-\varphi(P_1)]$ and
    \begin{align*}
        (\varphi(Q_1)[I-\varphi(P_2)] + \varphi(Q_2)[I-\varphi(P_1)) \bot
        ([\varphi(P_1) \vee \varphi(Q_1)] \wedge [\varphi(P_2) \vee \varphi(Q_2)]),
    \end{align*}
    we have
    \begin{align*}
        2c - \tau(P_1P_2) &= \tau(\varphi(P_1) \vee \varphi(P_2)) \\
        &\leq \tau([\varphi(P_1) \vee \varphi(Q_1)] \wedge
        [\varphi(P_2) \vee \varphi(Q_2)] \\
        &\leq \tau(\varphi(Q_1) + \varphi(Q_2))
        - \tau(\varphi(Q_1)[I-\varphi(P_2)] + \varphi(Q_2)[I-\varphi(P_1)])\\
        &=2c - \tau(P_1P_2).
    \end{align*}
     Therefore, $\varphi(P_1) \vee \varphi(P_2) = [\varphi(P_1) \vee \varphi(Q_1)] \wedge
        [\varphi(P_2) \vee \varphi(Q_2)]$.
    Note that $Q \leq P_1 \vee P_2 \leq [P_1 \vee Q_1] \wedge [P_2 \vee Q_2]$. And the
    assumption implies that
    \begin{align*}
        \varphi(Q) \leq  [\varphi(P_1) \vee \varphi(Q_1)] \wedge
        [\varphi(P_2) \vee \varphi(Q_2)] = \varphi(P_1) \vee \varphi(P_2).
    \end{align*}
\end{proof}

\begin{theorem}\label{include_two_case}
    Let $P_1, P_2, Q \in \PPP_c$.
    Assume that $P_1 P_2 = P_2 P_1$. Then $Q \leq P_1 \vee P_2$ implies
    $\varphi(Q) \leq \varphi(P_1) \vee \varphi(P_2)$.
\end{theorem}

\begin{proof}
    By \cref{include_comm_ind_step}, we assume that $0 \neq c_0 = \tau(I) - 2c < c$.
    It is clear that $\PPP_{c_0} \neq \emptyset$. Then \cref{include_comm_induction} implies the
    result.
\end{proof}

Let $2 \leq m \in \N$ such that $\PPP_{c/m} \neq \emptyset$.

\begin{lemma}\label{decom_exist}
    For every $Q \in \PPP_{c/m}$, there exist $m+1$ mutually commutative projections $P_i\in\PPP_c$,
    $i = 0, \ldots, m$,
    such that
    \begin{align*}
        Q = \frac{1}{m}[P_1 + \cdots +P_m - (m-1)P_{0}].
    \end{align*}
   \end{lemma}

\begin{proof}
    Note that $2c \leq \tau(I)$, there exists a family of orthogonal projections $\{Q_i\}_{i=0}^{m}
    \subset \PPP_{c/m}$ such that $Q = Q_0$. Let $P_{i} = \sum_{j=0}^{m}Q_j - Q_i$.
    Then we have
    \begin{align*}
        P_1 + \cdots +P_m - (m-1)P_{0} = m Q_0
    \end{align*}
     and $P_i\in\PPP_c$, $i=0, \ldots, m$.
\end{proof}

\begin{lemma}\label{decom_descri}
    Let $Q = \frac{1}{m}[P_1 + \cdots +P_m - (m-1)P_{0}] \in \PPP_{c/m}$, where
    $\{P_i\}_{i=0}^{m}$ is a family of projections in $\PPP_{c}$.
    Then there exists a family of orthogonal projections $\{Q_i\}_{i=0}^{m}
    \subset \PPP_{c/m}$ such that $Q = Q_0$ and $P_{i} = \sum_{j=0}^{m}Q_j - Q_i$.
\end{lemma}

\begin{proof}
   Note that
   \begin{align*}
       c = m\tau(Q) = \tau(QP_1) + \cdots +\tau(QP_m) - (m-1)\tau(QP_0) \leq c - (m-1)\tau(QP_0).
   \end{align*}
    This implies that $Q \leq P_i$, $i = 1, \ldots, m$ and $Q \bot P_{0}$.
    Let $P_i = Q + P_i'$, $i = 1, \ldots, m$. Then we have $P_1' + \cdots + P_m' = (m-1)P_0$ and
    \begin{align*}
        (m-1)c =(m-1)\tau(P_0) =  \tau(P_1' P_0) + \cdots + \tau(P_m'P_0).
    \end{align*}
    Since $\tau(P_i') = (m-1)c/m$, $P_i' \leq P_0$, $i = 1, \ldots m$.
    Let $Q_i = P_0 - P_i'$, $i = 1, \ldots m$. Note that $\tau(Q_i) = c/m$,
    \begin{align*}
        \frac{(m-1)c}{m} = (m-1)\tau(Q_i) = \sum_{j \neq i}\tau(Q_iP_j') \leq  \frac{(m-1)c}{m}.
    \end{align*}
    We have $Q_i \leq P_j'$ if $i \neq j$, $i, j \in \{1, \ldots, m\}$. For $i \neq j$,
    \begin{align*}
        (m-1)\tau(Q_iQ_j) = \sum_{k=1}^{m} \tau(Q_i P_k' Q_j) = (m-2)\tau(Q_iQ_j).
    \end{align*}
    Therefore, $Q_i \bot Q_j$ for $i \neq j$ and $P_{i} = \sum_{j=0}^{m}Q_j - Q_i$.
\end{proof}

\begin{lemma}\label{def_half_lemma}
    Let $Q = \frac{1}{m}[P_1 + \cdots +P_m - (m-1)P_{0}] \in \PPP_{c/m}$, where
    $\{P_i\}_{i=0}^{m}$ is a family of projections in $\PPP_{c}$. Then
    $\frac{1}{m}[\varphi(P_1) + \cdots + \varphi(P_m) - (m-1)\varphi(P_{0})] \in \PPP_{c/m}$.
\end{lemma}

\begin{proof}
    By \cref{decom_descri}, we assume that $P_{i} = \sum_{j=0}^{m}Q_j - Q_i$ where
    $\{Q_i\}_{i=0}^{m} \subset \PPP_{c/m}$ is a family of orthogonal projections
    such that $Q = Q_0$. In particular, $P_i P_j = P_j P_i$. Thus $\varphi(P_i)\varphi(P_j)
    = \varphi(P_j)\varphi(P_i)$ by \cref{P_Q_comm_general_case}. Note that
    \begin{align*}
        \tau(\varphi(P_i)\varphi(P_j)) = \tau(P_iP_j) = \frac{(m-1)c}{m}, \quad i \neq j.
    \end{align*}
    Also \cref{include_two_case} implies that $\varphi(P_i) \leq \varphi(P_{j}) + \varphi(P_{k})
    - \varphi(P_j)\varphi(P_k)$ if $j \neq k$.
    Let $i, j \in \{1, \ldots m\}$. If $i \neq j$ and $i\cdot j\neq 0$,
    \begin{align*}
        c = \tau(\varphi(P_0)) &= \tau(\varphi(P_0)\varphi(P_i))
        +\tau(\varphi(P_0)\varphi(P_j))- \tau(\varphi(P_0)\varphi(P_i)\varphi(P_j)) \\
        &= \frac{2(m-1)c}{m} - \tau(\varphi(P_0)\varphi(P_i)\varphi(P_j)).
    \end{align*}
    We have $\tau(\varphi(P_0)\varphi(P_i)\varphi(P_j))) = \frac{(m-2)c}{m}$.
    Let $E_i = \varphi(P_0) - \varphi(P_0)\varphi(P_i)$, $i = 1, \ldots m$.
    Note that $\tau(E_i) = c/m$ and
    \begin{align*}
        \tau(E_i E_j) &= \tau([\varphi(P_0) - \varphi(P_0)\varphi(P_i)][
            \varphi(P_0) - \varphi(P_0)\varphi(P_j)])\\
        &= c - \frac{2(m-1)c}{m} + \frac{(m-2)c}{m} = 0, \quad i \neq j.
    \end{align*}
    Thus $E_i \bot E_j$ and $\varphi(P_0) = \sum_{i=1}^m E_i$. For $i, j \in \{1, \ldots, m\}$ and $i \neq j$,
    \begin{align*}
        \tau(\varphi(P_i) E_j) = \tau(\varphi(P_i)\varphi(P_0))
        - \tau(\varphi(P_0)\varphi(P_i)\varphi(P_j)) = \frac{c}{m}.
    \end{align*}
    Therefore, $E_j \leq \varphi(P_i)$ and $\varphi(P_i) = F_i +
    \sum_{j=1}^m E_j - E_i$, where $F_i \in \PPP_{c/m}$ and $F_i \varphi(P_0) = 0$.
    Note that
    \begin{align*}
        \frac{(m-1)c}{m} =  \tau(\varphi(P_i)\varphi(P_j)) = \tau(F_iF_j) + \frac{(m-2)c}{m}(i\neq j),
    \end{align*}
    we have $F_i = F_j$. Let $E_0 = F_1$, we have $\varphi(P_i) = \sum_{i=0}^m E_i - E_i$ and
    $\frac{1}{m}[\varphi(P_1) + \cdots + \varphi(P_m) - (m-1)\varphi(P_{0})] = E_0 \in \PPP_{c/m}$.
\end{proof}

\begin{lemma}\label{def_half_eq}
   Let $Q = \frac{1}{m}[P_1 + \cdots +P_m - (m-1)P_{0}]$ and
   $Q' = \frac{1}{m}[P_1' + \cdots +P_m' - (m-1)P_{0}']\in \PPP_{c/m}$, where
    $\{P_i\}_{i=0}^{m}$ and $\{P_i'\}_{i=0}^{m}$ are two families of projections
    in $\PPP_c$. Then
    \begin{equation}\label{pre_tran_pro}
    \begin{aligned}
        &\tau([\varphi(P_1) + \cdots +\varphi(P_m) - (m-1)\varphi(P_{0})]
        [\varphi(P_1') + \cdots +\varphi(P_m') - (m-1)\varphi(P_{0}')]\\
        =& \tau([P_1 + \cdots +P_m - (m-1)P_{0}][P_1' + \cdots +P_m' - (m-1)P_{0}']).
    \end{aligned}
    \end{equation}
    In particular, if $Q = Q'$ then
    \begin{align*}
        \frac{1}{m}[\varphi(P_1) + \cdots +\varphi(P_m) - (m-1)\varphi(P_{0})]
        =\frac{1}{m}[\varphi(P_1') + \cdots + \varphi(P_m') - (m-1)\varphi(P_0')].
    \end{align*}
\end{lemma}

\begin{proof}
    It is clear that \cref{pre_tran_pro} holds. If $Q = Q'$, \cref{def_half_lemma}
    implies that $\frac{1}{m}[\varphi(P_1) + \cdots +\varphi(P_m) - (m-1)\varphi(P_{0})]
        =\frac{1}{m}[\varphi(P_1') + \cdots + \varphi(P_m') - (m-1)\varphi(P_0')]$.
\end{proof}

We are now in a position to prove \cref{main_result}.

\section{Proof of \cref{main_result}}\label{main}

\begin{proof}[\bf{Proof of \cref{main_result}}]
    We first assume that $c \in (0, \tau(I)/2)$ and claim that there is a subset
    $S$ of $[0, c]$ and a map $\Phi: \PPP_S \rightarrow \PPP_S$
    satisfying the following conditions:
    \begin{enumerate}
        \item $c \in S$ and every projection $P \in \PPP$ is a sum of a family
            of pairwise orthogonal projections in $\PPP_S$,
        \item $\Phi$ preserves the transition probablility and $\Phi|_{\PPP_c} = \varphi$.
    \end{enumerate}
    
    If the claim holds, we can prove the theorem in this case by invoking \cref{thm_lemma}.
    To prove the claim, we consider the following two cases.

    Case I: $\AAA$ is a type I factor. If $\PPP_c$ is the set of minimal projections of $\AAA$, then
    $S = \{0, c\}$ and $\varphi$ satisfy the condition (1) and (2).

    Otherwise, there exists $2 \leq m \in \N$ such that $\PPP_{c/m}$ is the set of minimal projections
    of $\AAA$. Let $S = \{0, c/m, c\}$ and
    \begin{align*}
        \Phi: P \in \PPP_{S} \mapsto
        \begin{cases}
            \varphi(P), & P \in \PPP_c\\
            \varphi_1(P), & P \in \PPP_{c/m} \\
            0, & 0 \in \PPP_{0}
        \end{cases},
    \end{align*}
    where $\varphi_1$ is defined by
\begin{align*}
    \varphi_1 : \frac{1}{m}[P_1 + \cdots +P_m - (m-1)P_{0}] \mapsto
        \frac{1}{m}[\varphi(P_1) + \cdots + \varphi(P_m) - (m-1)\varphi(P_{0})],
\end{align*}
where $\{P_i\}$ is a familiy of mutually commutative projections in $\PPP_c$.
    By \cref{decom_exist}, \cref{def_half_lemma} and \cref{def_half_eq},
$S$ and $\Phi$ satisfy the conditions.

    Case II: $\AAA$ is a type II factor. Let $S = \{0\} \cup \{\frac{c}{2^n}: n = 0, 1, 2, \ldots\}$.
    Then it is easy to see that $\PPP_S$ satisfies condtion (1).
    Let $\varphi_0 = \varphi$. We can define a map
    $\varphi_n: \PPP_{c/2^n} \rightarrow \PPP_{c/2^n}$ which
    preserves the transition probability for every $n \in \N$ inductively by
    \begin{align*}
        \varphi_n: \frac{1}{2}(P_1 + P_2 - P_0) \in \PPP_{c/{2^{n}}} \mapsto
        \frac{1}{2}(\varphi_{n-1}(P_1) + \varphi_{n-1}(P_2) - \varphi_{n-1}(P_0)) \in \PPP_{c/{2^{n}}},
    \end{align*}
    where $\{P_i\}_{i=0}^{2}$ is a familiy of mutually commutative projections in $\PPP_{c/2^{n-1}}$.
    Let $\Phi$ be the map from $\PPP_S$ to $\PPP_S$ given by
    $\Phi(P) = \varphi_n(P)$ for $P \in \PPP_{c/2^n}$, $n = 0, 1, 2 \ldots$ and $\Phi(0) = 0$.
    Suppose that $Q \in \PPP_{c/2^{k}}$ and $Q' \in \PPP_{c/2^{l}}$, $k, l \in \N$,
    we have
    \begin{align*}
        \varphi_{k}(Q) = \frac{1}{2^{k}}\sum_{i}\lambda_i\varphi(P_i), \quad
        \varphi_{l}(Q') = \frac{1}{2^{l}}\sum_{j}\beta_j\varphi(P_j'),
    \end{align*}
    where $\{P_i\}$ and $\{P_j'\}$ are two families of projections in
    $\PPP_c$ and $\lambda_i$, $\beta_j \in \{1, -1\}$ such that
    $Q = \frac{1}{2^{k}}\sum_{i}\lambda_i P_i$ and $Q' = \frac{1}{2^{l}}\sum_j \lambda_j P_j'$.
    Then it is clear that $\tau(\Phi(Q)\Phi(Q')) = \tau(QQ')$.

    Now we assume that $\tau(I) < 2c$. Without loss of generality, we assume that $\tau(I) = 1$.
    It is clear that
\begin{align*}
    \psi: P \in \PPP_{1-c} \mapsto I-\varphi(I-P) \in \PPP_{1-c}
\end{align*}
is a map preserving the transition probability. Note that $2(1-c) < \tau(I)$. The discussion above implies
there exists a projection $E \in \PPP$ and a $\sigma$-weak continuous *-homomorphims
$\sigma_1:\AAA \rightarrow E \AAA E$ and
    a $\sigma$-weak continuous *-anti-homomorphism $\sigma_2: \AAA \rightarrow (I-E)\AAA(I-E)$ such that
    $\psi(P) = \sigma_1(P) + \sigma_2(P)$ for every $P \in \PPP_{1-c}$ and $\sigma_1$ or $\sigma_2$ is
    zero map if $\AAA$ is a type I factor. Therefore
    \begin{align*}
        \varphi(Q) = I- \psi(I-Q) = I - \sigma_1(I-Q) - \sigma_2(I-Q) = \sigma_1(Q) + \sigma_2(Q),
        \quad \forall Q \in \PPP_c.
    \end{align*}
\end{proof}

\begin{corollary}
    Let $\AAA$ be a semifinite factor and $c \in (0, \tau(I)) \setminus \{\tau(I)/2\}$.
    If $\varphi:\PPP_{c} \rightarrow \PPP_{c}$ is
    a surjective map preserving the transition probability, then there exists
    a *-automorphism or a *-anti-automorphism $\sigma$ of $\AAA$ such that,
    $\varphi(P) = \sigma(P)$ for every $P \in \PPP_{S}$.
\end{corollary}

\begin{proof}
    If $\PPP_c = \emptyset$, the corollary holds trivially. Assume that $\PPP_c \neq \emptyset$.
    Note that if $A \in \AAA$ commutes with every $P \in \PPP_c$, then $A$ is in the center of $\AAA$.
    Since $\AAA$ is a factor, \cref{main_result} implies the result.
\end{proof}



\begin{thebibliography}{}
    \bibitem{AS}
    Alfsen, E. M.,    Shultz, F. W.:
         State Spaces of Operator Algebras: Basic Theory, Orientations, and C*-products,
        Mathematics: Theory \& Applications, Birkh\"{a}user,  Basel(2001)

\bibitem{BMS}
   Bracci, L., Morchio, G., Strocchi F.:
   Wigner's theorem on symmetries in indefinite metric spaces,
    Comm. Math. Phys. 41, 289-299(1975)

\bibitem{BB}
   Blackadar, B.:
  Operator Algebras, Theory of C$^*$-Algebras and von Neumann Algebras,
    Encyclopaedia of Mathematical Sciences, Vol. 122, Springer-Verlag Berlin Heidelberg (2006).

\bibitem{BW}
  Bunce,  L. J., Wright,  D. M.:
     The Mackey-Gleason problem,
    Bull. Am. Math. Soc. 26, 288-293(1992)

\bibitem{BJM}
    Botelho, F.,  Jamison, J., Moln\'{a}r, L.:
    Surjective isometries on Grassmann spaces,
    J. Funct. Anal. 265, 2226-2238(2013).

\bibitem{CW}
     Chevalier G.:
         Wigner's theorem and its generalizations,
        in: Handbook of Quantum Logic and Quantum Structures, pp. 429-475,  Elsevier Sci. B.V., Amsterdam, 2007.

\bibitem{MG}
     Gy\"{o}ry, M.:
    {\em Transformations on the set of all n-dimensional subspaces of a Hilbert space preserving
        orthogonality}, Publ. Math. Debrecen 65, 233-242(2004)

\bibitem{GS}
    Geh\'{e}r, G. P.,  \u{S}emrl P.:
    Isometries of Grassmann spaces,
    J. Funct. Anal. 270, 1585-1601(2016)

\bibitem{G2017}
  Geh\'{e}r,  G. P.:
     Wigner's theorem on Grassmann spaces,
    J. Funct. Anal. 273, 2994-3001(2017).


\bibitem{KRII}
     Kadison, R. V.,  Ringrose, J. R.:
     Fundamentals of the theory of operator algebras, Vol.\ I, II,
    Graduate Studies in Mathematics, Vol. 15, 21,
    American Mathematical Society, 1997.

\bibitem{M0}
    Moln\'{a}r L.:
     The set of automorphisms of $\BBB(\HHH)$ is topologically reflexive in
        $\BBB(\BBB(\HHH))$,
    Studia Math. 122, 183-193(1997)

\bibitem{M1}
    Moln\'{a}r L.:
    Wigner-Type Theorem on Symmetry Transformations in Type II Factors,
    International Journal of Theoretical Physics 39, 1463(2000)

\bibitem{M4}
   Moln\'{a}r L.:
    Generalization of Wigner's unitary-antiunitary theorem for indefinite inner product spaces,
    Comm. Math. Phys. 201, 785-791(2000)

\bibitem{M2}
   Moln\'{a}r L.:
    Transformations on the Set of All $n$-Dimensional Subspaces of a Hilbert Space
        Preserving Principal Angles,
    Commun. Math. Phys. 217, 409-421(2001)

\bibitem{M5}
    Moln\'{a}r L.:
    Orthogonality preserving transformations on indefinite inner product spaces: generalization
    of Uhlhorn's version of Wigner's theorem,
    J. Funct. Anal. 194, 248-262(2002)

\bibitem{M3}
   Moln\'{a}r L.:
    Selected Preserver Problems on Algebraic Structures of Linear Operators and on Function Spaces, Lecture Notes in Mathematics (1895),
    Springer-Verlag Berlin Heidelberg (2007)

\bibitem{M6}
  Moln\'{a}r L.:
    Maps on the $n$-dimensional subspaces of a Hilbert space preserving principal angles,
    Proc. Amer. Math. Soc. 136, 3205-3209(2008)

\bibitem{S0}
    Moln\'{a}r L.:
    Generalized symmetry transformations on quaternionic indefinite inner product spaces:
    an extension of quaternionic version of Wigner's theorem,
    Comm. Math. Phys. 242, 579-584(2003)

\bibitem{S}
    \u{S}emrl, P.:
    Orthogonality preserving transformations on the set of $n$-dimensional subspaces of a Hilbert space,
    Illinois Journal of Mathematics, 48, 567-573(2004)

\bibitem{ES}
    St\o rmer, E.:
    On the Jordan structure of C$^*$-algebras,
    Trans. Am. Math. Soc. 120, 438-447(1965)

\bibitem{ES2}
  St\o rmer, E.:
      Positive Maps Which Map the Set of Rank $K$ Projections onto Itself, Positivity, 21(1), 1-3(2016)

\bibitem{Tak}
    Takesaki, M.:
    Theory of Operator Algebras I, II, III,
    Encyclopaedia of Mathematical Sciences, Springer-Verlag Berlin Heidelberg (2002, 2003).

\bibitem{UH}
  Uhlhorn U.:
  Representation of symmetry transformations in quantum mechanics, Ark. Fys. 23, 307-340(1963)

\bibitem{W}
   Wigner, E. P.: Gruppentheorie und ihre Anwendung auf die Quanten mechanik der Atomspektren,
   Fredrk Vieweg und Sohn  (1931)

\end{thebibliography}
\end{document}